\documentclass[11pt]{amsart}
\usepackage{amsmath, amssymb, amsthm}
\usepackage{graphicx}
\usepackage{xcolor}
\usepackage{booktabs}
\usepackage{amsmath,amssymb, amsthm}
\usepackage[all]{xy}
\usepackage[inline]{enumitem}
\usepackage{tikz}
\usepackage{tikz-cd}
\usepackage[a4paper,left=25mm,right=25mm,top=25mm,bottom=25mm]{geometry}

\usepackage{quiver}
\usepackage[most]{tcolorbox}
\usepackage{mathtools}
\usepackage[colorlinks=true,linkcolor=blue,citecolor=blue,urlcolor=blue]{hyperref}

\theoremstyle{plain}
 \newtheorem{theorem}{Theorem}[section]
 \newtheorem{proposition}[theorem]{Proposition}
 \newtheorem{corollary}[theorem]{Corollary}
 
 \newtheorem{lemma}[theorem]{Lemma}

\theoremstyle{definition}

\newtheorem{example}[theorem]{Example}

\hypersetup{
  colorlinks=true,
  linkcolor=blue!60!black,
  citecolor=blue!60!black,
  urlcolor=blue!60!black
}
 
 \newcommand{\ZZ}{{\mathbb{Z}}}
 
 \newcommand{\RR}{{\mathbb{R}}}
 
 \newcommand{\HH}{H^{*}} 
 \newcommand{\Z}{\mathbb{Z}}
 \newcommand{\R}{\mathbb{R}}
 \newcommand{\Ker}{\mathop{\rm Ker}\nolimits}
 \newcommand{\im}{\mathop{\rm Im}\nolimits}

 \newcommand{\hdim}{\mathop{\rm hdim}\nolimits}
 
 \newcommand{\cl}{\mathop{\rm cl}\nolimits}
 \newcommand{\ev}{\mathop{\rm ev}\nolimits}
 \newcommand{\cat}{\mathop{\rm cat}\nolimits}
 
 \newcommand{\secat}{\mathop{\rm secat}\nolimits}
 
 \newcommand{\tc}{\mathop{\rm tc}\nolimits}
 \newcommand{\pr}{\mathop{\rm pr}\nolimits}
 
 \newcommand{\TC}{\mathop{\rm TC}\nolimits}

 \newcommand{\id}{\mathrm{id}}

 \newcommand{\cd}{\mathrm{cd}}

 \DeclareMathOperator{\Aut}{Aut}

\newcommand{\bigrt}{\mathop{\rtimes}\limits}

\newcommand{\rt}{\rtimes}

\newcommand{\pullbackcorner}[1][dr]{\save*!/#1-1.5pc/#1:(-1,1)@^{|-} \restore}

\begin{document}

\title{Parametrized LS category and group actions}
\author{Urban Ogrinec \and Petar Pavešić}

\address{{Urban Ogrinec,
Institute for Mathematics, Physics and Mechanics, Ljubljana, Slovenia; 
Faculty of Mathematics and Physics, University of Ljubljana, Ljubljana, Slovenia}}
\email{urban.ogrinec@fmf.uni-lj.si}

\address{{Petar Pavešić,
Institute for Mathematics, Physics and Mechanics, Ljubljana, Slovenia; 
Faculty of Mathematics and Physics, University of Ljubljana, Ljubljana, Slovenia}}
\email{petar.pavesic@fmf.uni-lj.si}

\thanks{$^{*}$ Supported by the Slovenian Research Agency program P1-0292 and grant J1-4001}

\date{\today}

\begin{abstract}
We introduce a new invariant associated to an action of a group $G$ on a (possibly non-commutative)
group $K$. The invariant is based on the $G$-parametrized LS-category of the semi-direct product
$K\rtimes G$ but it also depends on a choice of a derivation $d\colon G\to K$. 
The invariant is thus denoted $\tc(d)$ and it takes integer values between 
$\cd(K)$ and $\cd(K)+\cd(G)$ (where $\cd(G)$ stands for the cohomological dimension of a group). We develop 
a general theory and then compute its values for a number of important examples, including 
fundamental groups of mapping tori, of pure braid groups and of iterated semi-direct 
products of free groups. To achieve this, we relate $\tc(d)$ to the cohomological dimension
of its factors, to the sectional category of the inclusion of $G$ in $K\rtimes G$, and to 
the properties of torsors that are associated to the derivation $d$.

\ \\[3mm]
\end{abstract}

\maketitle

\section{Introduction}\label{sec:Introduction}

In this work, we discuss a new LS-category type invariant associated to an action
of a group on another group. The celebrated Eilenberg-Ganea Theorem 
\cite{EilenbergGanea1957} states that cohomological dimension $\cd(G)$ of a group $G$ equals $\cat(BG)$,
the LS-category of its classifying space, In other words, the theorem relates projective resolutions of 
$\ZZ G$-modules with covers of the classifying space $BG$ by subspaces that contract within $BG$ to a point. 

Given an action of $G$ on another group $K$,  one can construct the corresponding
semi-direct product $K\rtimes G$. By applying the classifying space functor to the short
exact sequence 
$$1\to K\to K\rtimes G\to G\to 1$$
we obtain a fibration sequence
$$ BK\hookrightarrow B(K\rtimes G) \to BG. $$
It can be viewed as a family of copies of $BK$ parametrized by points in $BG$ in a manner that encodes
the action of $G$ on $K$. 
The difference between the $BG$-parametrized LS-category $\cat_{BG}(B(K\rtimes G))$ and 
$\cat(BK)$ is a measure of the complexity of the action of $G$ on $K$. 

The precise definition of our invariant is a bit more involved, because parametrized category also depends 
on the choice 
of a section in the above fibration. In group-theoretic terms, this section is determined by a choice
of a derivation $d\colon G\to K$. Although our invariant is defined as a parametrized category of a fibration,
we will denote it as $\tc(d)$ and call \emph{it the (topological) complexity} of $d$. The reason is that 
in the special 
case where $G$ acts trivially on $K$ and a derivation $d\colon G\to K$ is simply a homomorphism, 
then this invariant equals (somewhat surprisingly) the topological complexity (in the sense of \cite{Pavesic2019a})
of the classifying map $Bd\colon BG\to BK$.

\subsection{Prior work}

In recent years several articles have  discussed LS category-type invariants associated with homomorphisms 
and group actions. In the papers 
\cite{DranishnikovKuanyshov2023, Kuanyshov2023, Kuanyshov2024a,Kuanyshov2024b,DeSahaDranishnikov2024}
authored (in different combinations) by De Saha, Dranishnikov and Kuanyshov cohomological dimension
of a homomorphism $\varphi \colon G\to H$ is defined as the maximal $k$ for which $\varphi$ induces
a non-trivial homomorphism in group cohomology with some coefficients. Furthermore, $\cat(\varphi)$ and
$\TC(\varphi)$ are defined respectively as the category and the topological complexity of the 
map $B\varphi\colon BG\to BH$. The main results of the papers determine various conditions under which 
Eilenberg-Ganea theorem extends as $\cd(\varphi)=\cat(\varphi)$ or $\TC(\varphi)=2\,\cd(\varphi)$.

In a similar vein, Espinosa Baro, Farber, Mescher and Oprea \cite{EBFMO} 
examined the sectional category of the covering space projection $BH\to BG$ associated to 
a subgroup inclusion $H\hookrightarrow G$ and derive a lower bound in terms of $\cd(G)$ and
the cohomological dimensions of some subgroups of $H$.

Grant, Meir and Patchkoria \cite{GrantMeirPatchkoria2022} considered a group $G$ acting on 
another group $\pi$ and the resulting semidirect product $\pi\rtimes G$. They define 
equivariant versions of cohomological dimension $\cd_G(\pi)$, geometric dimension
$\mathrm{gd}_G(\pi)$ and category $\cat_G(\pi)$ in terms of corresponding 
invariants for $\pi\rtimes G$. Similarly as above, their main results concern 
assumptions under which those invariants coincide.

\subsection{Our contribution}

Although our terminology is similar to that used in the above-mentioned papers, 
our approach is different and more in the 
spirit of parametrized topological complexity as introduced in \cite{CohenFarberWeinberger2021}. 
Assume that $G$ is a group acting on another group $K$ by 
automorphisms. Then we have the semi-direct product $K\rtimes G$ and the projection
$p\colon K\rtimes G\to G$. By applying the classifying space functor we obtain a fibration
$Bp\colon B(K\rtimes G)\to BG$ with fibre $BK$. Homotopy invariants of this fibration 
(e.g., the parametrized LS-category or the parametrized topological complexity) reflect 
the complexity of the action of $G$ on $K$ and of the semi-direct product. 
Here we must take into account a technical point: in order to define $BG$-parametrized 
LS.category of $B(K\rtimes G)$ we must fix a section to $Bp$. It is known that every section 
$G\to K\rtimes G$ to $p$ is determined by a derivation $d\colon G\to K$. Thus, we pick a derivation
$d\colon G\to K$ and define $\tc(d)$ as the parametrized LS-category of $B(K\rtimes G)$ over 
$BG$, relative to a section determined by $d$. 

We derive several basic estimates for $\tc(d)$ in terms of the cohomological dimension of the 
groups $G$ and $K$ and in terms of nilpotency of  a certain ideal in the cohomology of the group
$K\rtimes G$. We discuss in considerable detail the case of the trivial action of $G$ on $K$, and show
how it is related to the topological complexity (as defined in \cite{Pavesic2017}) of the map $Bd$.
Returning to the general case, we obtain a very useful alternative description of $\tc(d)$
in terms of the sectional category of an associated covering projection. We study $\tc(d)$ for different
choices of derivations and show that for abelian $K$, the value of $\tc(d)$ depends only on the action
of $G$ on $K$. In the non-abelian case we find an interesting relation with torsors and show that 
$\tc(d)=\tc(d')$ if $d$ and $d'$ represent the same element in the 
set of $\operatorname{Inn(K)}$-torsors for which the Brauer obstruction vanish. In addition, we give 
some estimate for the complexity of derivations that arise in iterated semi-direct products.

In the second part of the paper we apply the general theory to compute the complexity of 
several interesting semi-direct products and corresponding derivations. This includes semi-direct
product decomposition of fundamental groups of mapping tori, standard semi-direct decompositions
of pure braid groups, and iterated semi-direct products of free groups. In the majority of the mentioned examples
we are able to compute the precise values of the invariant.

\subsection{Outline of the paper}
In the following section, we recall some preliminary notions that will be used in the rest of the paper: 
classifying spaces of discrete groups, semi-direct products, and the parametrized LS-category. 
In Section 3 we define the topological complexity of a derivation and derive its main properties.
This includes a detailed examination of the case where the action is trivial (and the semi-direct 
product is actually a direct product) and derivation of some general upper and lower bounds 
based on the cohomological dimensions of the groups involved and on the cup-length of 
certain cohomology rings. Furthermore, we give an alternative description in terms of 
the sectional category of certain covering spaces, derive some new lower bounds that are not based
on cup-length, and study when the topological complexity depends only on the semi-direct product
structure and not on a choice of a derivation. In Section 4 we present extensive computations 
for some important classes of semi-direct products like the fundamental groups of mapping tori, 
pure braid groups and iterated products of free groups.

\section{Preliminaries}\label{sec:Preliminaries}

\subsection{Classifying spaces}
All groups considered in this work are discrete. Given a group $G$ we denote by $BG$ its classifying 
space,
which is, of course, an Eilenberg-MacLanes space of type $K(G,1)$. The space $BG$ is defined only up to 
homotopy and the precise geometry of $BG$ depends on a chosen construction. In specific cases, we will 
opt for
a construction that fits our needs. For example, given a short exact sequence of groups
$$1\to K\to E\to G\to 1$$ 
we may use the simplicial set construction, which is functorial and yields a fibration sequence
$$ BK\hookrightarrow BE \to BG. $$
Alternatively, for a subgroup $K\le G$, we may use the Borel construction to obtain a covering projection 
$BK\to BG$ with fibre $G/K$. In some examples, we will also use particular spaces to model $BG$ like the 
configuration spaces of points in the plane for the pure braid group, or the mapping tori for semi-direct
products with $\ZZ$.

\subsection{Semi-direct products}
Semi-direct products will play a prominent role, so we recall some main facts and refer the reader to 
\cite[Ch. IV]{Brown} for more details. Let $G$ and $K$ be two groups, where the operation in $G$ 
is written as multiplication, and the operation in $K$ is written as addition, even when $K$ 
is \textit{not} abelian. An action of $G$ on $K$ is given by a homomorphism $\varphi\colon G\to\Aut(K)$. 
If the action $\varphi$ is clear from the context we will write $g\cdot k$ instead of 
$[\varphi(g)](k)$. Moreover, if $K$ is abelian, we will occasionally 
say that $K$ is a $\ZZ G$-module.
The \emph{semi-direct product} $K\rtimes G$ is a group whose 
underlying set is $K\times G$ and the multiplication is given by
$$(k,g)\cdot (k',g'):=(k+[\varphi(g)](k'),g\cdot g').$$
The semi-direct product fits into a short exact sequence 
$$1 \to K \xrightarrow{i} K \rtimes G \xrightarrow{p} G \to 1,$$
where $i(k)=(k,1)$ and $p(k,g)=g$. There is a canonical splitting (right-inverse group homomorphism) 
for the map $p$, given by $s_0(g)=(0,g)$ but this is not the only possibility. 
In fact, given a function $d\colon G\to K$, then 
$s_d(g):=(d(g),g)$ is a splitting of $p$ if, and only if, $d$ is a \emph{derivation} (or crossed-homomorphism), 
i.e. it satisfies $d(g\cdot h)=d(g)+\varphi(g)(d(h))$. Note that if $G$ acts trivially on $K$, then derivations
$G\to K$ are just ordinary homomorphisms. If $K$ is a $\ZZ G$-module, then it is known that 
the first cohomology group $H^1(G;K)$ classifies derivations $G\to K$ modulo principal derivations
(see \cite[Prop. 2.3]{Brown}).

\subsection{Parametrized LS-category} In this section, we recall some basic definitions  
related to the (Lusternik-Schnirelmann) category of a space. Standard references are \cite{CLOT}
for the general theory, and \cite{Garcia-Calcines2014} for the parametrized (or fibrewise) versions.

Let $X$ be any space. A subset $U\subseteq X$ is said to be \emph{categorical} if the inclusion
map $U\hookrightarrow X$ is nul-homotopic. The \emph{LS category} of $X$, denoted $\cat(X)$, 
is the minimal $n$ such that 
there exists a cover $U_0,\ldots,U_n$ of $X$ by open categorical subsets. A more general 
concept is the sectional category of a fibration $p\colon X\to B$. A subset $U\subseteq B$ is 
\emph{sectionable} if there exists a section $s\colon U\to X$ for the projection $p$. 
The \emph{sectional category} of a fibration $p\colon X\to B$, denoted $\secat(p)$ 
is the minimal $n$ such that
there exists a cover $U_0,\ldots,U_n$ of $B$ by open sectionable subsets. 
The sectional category of an arbitrary map is defined as the sectional category of its fibrational replacement.

LS category can be expressed in terms of sectional category: given a point $s\in X$ let $P_sX$ denote 
the space of based paths $\mathcal{C}((I,0),(X,s))$ and by $\ev_1\colon P_sX \to X$ the evaluation map 
that to each based path assigns its final point. If $X$ is path-connected, one can easily check that 
$\cat(X)=\secat(\ev_1)$.

Next we extend the above concepts to a parametrized setting.
Let $B$ be any space. A \emph{fibrewise pointed} space over $B$ is a projection $p\colon X\to B$,
together with a section $s\colon B\to X$, such that $p\circ s=1_B$. For each $b\in B$ the preimage
$X_b:=p^{-1}(b)\subset X$ is called the \emph{fibre} of $p$ over $b$. We view a fibrewise pointed
space as a continuous family of fibres parametrized by points $b\in B$, together with 
a continuous choice of a base-point 
$s(b)$ at each fibre. If $p\colon X\to B$ is a fibration over a path-connected space $B$, 
then all fibres
are of the same homotopy type. In that case we will talk about a generic fibre $F$ of $p$. 

A subspace $U\subseteq X$ is \emph{$B$-categorical} if the inclusion $U\hookrightarrow X$ is 
fibrewise nul-homotopic, i.e. there exists a homotopy $H\colon U\times I\to X$, such that 
for all $x\in X$ and $t\in I$ we have $p(H(x,t))=p(x)$, $H(x,0)=x$ and $H(x,1)=s(p(x))$. 
The \emph{$B$-parametrized LS category} of a fibrewise pointed space $X$ over $B$, denoted $\cat_B(X)$, 
is the minimal $n$ such that there exists a cover $U_0,\ldots,U_n$ of $X$ by open 
$B$-categorical subsets.

Parametrized category can be also expressed in terms of sectional category. To this end we define 
the space 
$$P_BX:=\{\alpha\in\mathcal{C}(I,X)\mid p\circ\alpha=\mathrm{const}, \alpha(0)=s(p(\alpha(0))) \}.$$
In plain words, a path in $X$ is an element of $P_BX$ if its image is contained in a fibre and if its
initial point is contained in $s(B)$. There is a map $\ev_1\colon P_BX\to X$ that to each path
in $P_BX$ assigns its final point. In the following theorem we list some of the main properties of the 
parametrized LS category.

\begin{theorem}\label{thm:catB(X)}
Assume $p\colon X\to B$ is a fibration with a section $s\colon B\to X$ and a path-connected fibre $F$.
Then
\begin{enumerate}[wide, labelwidth=!]
\item $\ev_1\colon P_BX\to X$ is a fibration with fibre $\Omega F$.

\item $\cat_B(X)=\secat(\ev_1).$

\item $\cat(F)\le\cat_B(X)\le \cat(X)$

\item Let $H^*$ denote any cohomology theory, and let $s^*\colon H^*(X)\to H^*(B)$ be the 
homomorphism induced by the section $s\colon B\to X$. Then  
$\cat_B(X)\ge \cl(\Ker s^*)$, 
where $\cl(R)$ of a ring $R$ denotes its cup-length, i.e. the maximal number of factors from $R$
whose product is non-zero.
\end{enumerate}
\end{theorem}
\begin{proof}
(1) That $\ev_1$ is fibration follows from \cite[Appendix]{CohenFarberWeinberger2022}. 
Moreover, one can easily check that for $b\in B$ the fibre of $\ev_1$ over $s(b)\in X$ 
is precisely the space $\Omega X_b$ of loops in $X_b$ based at $s(b)$.

(2) Let $U\subseteq X$ and assume that there exists a section $h\colon U\to P_BX$, such that 
$\ev_1\circ h$ equals the inclusion of $U$ in $X$. Note that $(h(x))(1)=x$, therefore 
$p(h(x)(t))=p(x)$ for all $t\in I$ and $(h(x))(0)=s(p(x))$. Thus, if we define $H\colon U\times I\to X$ 
as $H(x,t):=(h(x))(1-t)$, we obtain a fibrewise nul-homotopy for $U$, which implies that $U$ is 
$B$-categorical. Conversely, if $U\subseteq X$ is $B$-categorical, then a fibrewise nul-homotopy of $U$
yields a section of $\ev_1$ over $U$. Thus we have a correspondence between $B$-categorical and 
$\ev_1$-sectionable subsets of $X$, which proves the assertion.

(3) Take some $b\in B$ and observe that the restriction of  $\ev_1\colon P_BX\to X$ to $X_b\subset X$ 
is precisely $\ev_1\colon P_{s(b)}X_b\to X_b$. Since sectional category decreases upon restriction 
we obtain the lower estimate $\cat(F)\le\cat_B(X)$. On the other hand, since 
$\ev_1\colon P_BX\to X$ is a fibration,  every categorical subset of $X$ admits a section, therefore 
$\cat_B(X)\le \cat(X)$.

(4) A standard lower estimate for sectional category of a fibration $p\colon E\to B$ is 
$\secat(p)\ge \cl(\Ker p^*)$. Let us consider the following commutative diagram
$$\xymatrix{
B\ar[rr]^c \ar[dr]_s & &P_BX \ar[dl]^{\ev_1}\\
& X
}$$
where the map $c$ assigns to every $b\in B$ the constant loop in $X$ at the point $s(b)$. 
The map $c$ is a homotopy equivalence, its homotopy inverse being given by the 
evaluation at the initial point $p\circ \ev_0\colon P_BX\to B$. By applying a cohomology
theory $H^*$ we obtain a commutative diagram
$$\xymatrix{
H^*(B)  & &H^*(P_BX) \ar[ll]_{c^*\ \ \cong}\\
& H^*(X) \ar[ur]_{\ev_1^*}\ar[ul]^{s^*}
}$$
which implies that $\Ker \ev_1^*=\Ker s^*$, therefore $\secat(\ev_1)\ge \cl(\Ker\ev_1^*)=\cl(\Ker s^*)$.
\end{proof}

\section{Topological complexity of a derivation}\label{sec:Topological complexity of a derivation}

In this section, we define and study a new invariant associated to derivations and semi-direct 
products. 

\subsection{Definition of $\tc(d)$}

Let $(G,\cdot)$ and $(K,+)$ be two groups, and let a homomorphism $\varphi\colon G\to \Aut(K)$ 
determine an action of $G$ on $K$. Furthermore, let $d\colon G\to K$ be a derivation relative to 
the action $\varphi$. Then we may 
construct the semi-direct product $K\rtimes G$ and a homomorphism $s_d\colon G\to K\rtimes G$
given by $s_d(g):=(d(g),g)$. Clearly, $s_d$ is a splitting for the projection 
$p\colon K\rtimes G\to G$. By applying a suitable classifying space construction, we obtain
a fibration sequence
$$BK\hookrightarrow B(K\rtimes G)\stackrel{Bp}{\longrightarrow} BG.$$
Moreover, the splitting $s_d$ determines a section $B(s_d)$ to the projection $Bp$. If we construct 
the space $P_{BG}B(K\rtimes G)$ with respect to this section, we obtain a fibrewise-pointed 
space over $BG$ 
and we may define the \emph{complexity of the derivation $d$} as
$$\tc(d):=\cat_{BG}(B(K\rtimes G))=\secat\big(\ev_1\colon P_{BG}B(K\rtimes G)\to B(K\rtimes G)\big).$$

\subsection{Trivial action} As an extended example, let us first consider the special 
case where $G$ acts trivially on $K$ 
(i.e., $\varphi(g)=1_K$ for all $g\in G$). In that case $K\rtimes G=K\times G$ 
 and the derivations $d\colon G\to K$ are actually homomorphisms. Pick a homomorphism $d\colon G\to K$ 
 and consider the corresponding section $s_d=(d,1)\colon G\to K\times G$ to the projection map 
 $p\colon K\times G\to G$. To compute $P_{BG}B(K\times G)$ we rely on the following lemma.

 \begin{lemma}
 Given a map $f\colon B\to F$ let $s=(1,f)\colon B\to B\times F$ be a section to the projection 
 $p\colon B\times F\to B$. Then the following diagram is a pull-back
 $$
\xymatrix{
P_B(B\times F) \ar[d]_{\ev_1} \pullbackcorner \ar[r] & F^I\ar[d]^{\ev_{0,1}} \\
B\times F \ar[r]_{f\times 1} & F\times F
}$$ 
\end{lemma}
\begin{proof}
 The result follows by direct computation:
 \begin{align*}
  P_B(B\times F) & =\{\alpha\colon I\to B\times F\mid \pr_B\alpha=\mathrm{const}_b, 
 \alpha(0)=(b,f(b))\}= \\
 & = \{(b,\alpha)\in B\times F^I|\alpha(0)=f(b)\}=\\
 & = \{((b,,x),\alpha)\in (B\times F)\times  F^I|\alpha(0)=f(b), \alpha(1)=x\}=\\
 & = (B\times F)\times_{F\times F} F^I
 \end{align*}
 Moreover, under the above identification, the map $\ev_1$ corresponds to the projection to $B\times F$.
\end{proof}
By applying the lemma to our specific situation, we obtain a pull-back diagram
$$
\xymatrix{
P_{BG}(BG\times BK) \ar[d]_{\ev_1} \pullbackcorner \ar[r] & (BK)^I\ar[d]^{\ev_{0,1}} \\
BG\times BK \ar[r]_{(Bd)\times 1} & BK\times BK
}$$ 
This turns out to be precisely the diagram used in the description of the topological complexity of 
the map $Bd$ in  \cite[Thm 4.9]{Pavesic2019a} (see also discussion after \cite[Thm 4.16]{Pavesic2019a}). 
Thus we obtain the following result:

\begin{proposition}\label{prop:trivial action}
If $G$ acts trivially on $K$ and if $d\colon G\to K$ is any derivation (homomorphism), then 
$$\tc(d)=\TC(Bd).$$   
\end{proposition}

Observe that by the homotopy invariance property $\TC(Bd)$ is usually called the topological complexity
of the homomorphism $d$. We should also mention that Scott \cite[Sec. 6]{Scott2022} defined a slightly different 
version of the topological complexity of a homomorphism; see \cite{Kuanyshov2024b} for some later developments
and \cite[Section 13.5.3]{Pavesic2024} for a comparison between the two concepts. 

In the following examples the standing assumption is that $G$ acts trivially on $K$.

\begin{example}\label{ex:trivial derivation}
If $d\colon G\to K$ is a trivial derivation, then $Bd$ is nulhomotopic, so Proposition \ref{prop:trivial action}
and \cite[Thm. 13.20(3)]{Pavesic2024} imply that $\tc(d)=\cat(BK)=\cd(K).$    
\end{example}

\begin{example}\label{ex:trivial action with abelian kernel}
If the group $K$ is abelian, then $BK$ is an $H$-space and \cite[Thm. 13.20(8)]{Pavesic2024}   
states that in this case $\TC(Bd)=\cat(BK)$. Thus, we obtain $\tc(d)=\TC(Bd)=\cat(BK)=\cd(K)$ for 
every derivation $d\colon G\to K$. We will extend this result to arbitrary actions in 
Corollary \ref{cor:K abelian}.
\end{example}

\begin{example}\label{ex:F2 identity}
We will see later more instances where different derivations often have the same complexity. 
For an explicit example where complexity depends on the choice of derivation, let $F_2$ denote the 
free group
of rank 2, so that $BF_2$ is the wedge of two circles. Furthermore, let $d_0\colon F_2\to F_2$ be 
the trivial 
derivation and let $d_1\colon F_2\to F_2$ be the identity homomorphism. Then we have
$$\tc(d_0)=\cat(BF_2)=1$$
as explained above, and
$$\tc(d_1)=\TC(1_{BF_2})=\TC(BF_2)=2.$$
\end{example}

\begin{example}\label{ex:Rudyak}
The last example can be extended to show that the range of values due to different 
derivations can be arbitrarily large. Let $G:=\ZZ^k\ast\ZZ^l$, where $0\le k\le l$ and let
$d_0,d_1\colon G\to G$ denote respectively the trivial and the identity derivation.
Then by the results of Rudyak \cite{Rudyak2016} we obtain 
$\tc(d_0)=\cat(G)=l$ and $\tc(d_1)=\TC(G)=l+k$.
\end{example}

\subsection{Basic estimates}
Before considering more complicated examples, we need to develop some general theory. 

\begin{theorem}\label{thm:cd estimates for tc}
Let $d\colon G\to K$ be a derivation relative to an action of $G$ on $K$. Then
$$\cd(K)\le\tc(d)\le \cd(K\rtimes G)\le\cd(K)+\cd(G).$$
\end{theorem}
\begin{proof}
The result follows immediately from Theorem \ref{thm:catB(X)}(3) and the Ganea-Eilenberg
equality $\cat(BG)=\cd(G)$. The right-hand side inequality  $\cd(K\rtimes G)\le\cd(K)+\cd(G).$
is also well-known (see \cite[Prop.VIII.2.4]{Brown}).
\end{proof}

Actually, under some mild assumptions, one has $\cd(K\rtimes G)=\cd(K)+\cd(G)$, 
see \cite[Thm. 5.5]{Bieri1977}. The above theorem gives a general interval containing
$\tc(d)$ which is valid for all derivations $d\colon G\to K$. The following result is more
specific and will often result in better estimates.

\begin{theorem}\label{thm:coho estimate for tc}
Let $d\colon G\to K$ be a derivation relative to an action of $G$ on $K$ and let 
$s_d\colon G\to K\rtimes G$ be given as $s_d(g):=(d(g),g)$. Pick any $K\rtimes G$-module
$M$ and denote by $s_d^*\colon H^*(K\rtimes G;M)\to H^*(G; M)$. Then
$$\tc(d)\ge \cl(\Ker s_d^*).$$
\end{theorem}
\begin{proof}
The result follows from Theorem \ref{thm:catB(X)}(4). It is sufficient to observe that group 
cohomology with coefficients in a $K\rtimes G$-module $M$ corresponds to cohomology of 
its classifying space with local coefficients.
\end{proof}

At this point we should mention that  the cup-product of elements $u\in H^i(G;M), v\in H^j(G;M)$
is, in general, an element $u\cup v\in H^{i+j}(G;M\otimes M)$. In practice, we will mostly 
consider integer coefficients and a more common form of a cup-product with values in $H^*(G;\ZZ).$

\subsection{Alternative description of $\tc(d)$.} 
As we have observed in the proof of Theorem \ref{thm:catB(X)}, the inclusion 
$B\hookrightarrow P_BX$ that to every $b\in B$ assigns the constant loop at $s(b)$ is a homotopy
equivalence. It follows that $P_{BG}B(K\rtimes G)\simeq BG$ and so the homotopy fibre 
of the map $B(s_d)\colon BG\to B(K\rtimes G)$ is discrete, which suggests that the fibration 
$B(s_d)$ could be replaced by a covering space. This leads to a more explicit description of 
$\tc(d)$:

\begin{theorem}\label{thm:tc=secat}
Let $d\colon G\to K$ be a derivation relative to an action of $G$ on $K$ and 
$s_d\colon G\to K\rtimes G$ the corresponding splitting of the semi-direct product.  
Denote by $\widetilde{B(K\rtimes G)}_{\im(s_d)}$ the covering of $B(K\rtimes G)$ determined by
the subgroup $s_d(G)\le K\rtimes G=\pi_1(B(K\rtimes G))$ and by 
$\pi_d\colon\widetilde{B(K\rtimes G)}_{\im(s_d)}\to B(K\rtimes G)$ the corresponding covering
projection. Then\\[2mm]
(1) $\pi_d$ is the covering projection over $B(K\rtimes G)$, whose fibre is the group $K$, viewed as 
a $K\rtimes G$-set with action given by $(k,g)\cdot k':=k+g\cdot k'-d(g)$.

(2) $\tc(d)=\secat(\pi_d).$
\end{theorem}
\begin{proof}
(1) Observe that $\pi_d$ is a regular covering if and only if the action of $G$ on $K$ is trivial. 
A covering over $B(K\rtimes G)$ is classified by  the action of the fundamental
group $K\times G$ on its fibre (cf. \cite[pp. 69-70]{Hatcher}). The fibre of $\pi_d$ is the set of cosets
of the group $K\rtimes G$ with respect to the subgroup $s_d(G)$. Every element of $K\rtimes G$ can 
be decomposed as 
$$(k,g)=(k-d(g),1)\cdot (d(g),g)=(k-d(g),1)\cdot s_d(g).$$
As a consequence, the function $(k,g)\mapsto k-d(g)$ induces a bijection between the fibre
$K\rtimes G/s_d(G)$ and $K$. Furthermore, the equality $(k,g)\cdot (k',1)=(k+g\cdot k,g)$ 
shows that with respect to this bijection, the action of $K\rtimes G$ on $K$ is given by 
$(k,g)\cdot k':=k+g\cdot k'-d(g)$.

(2) Consider the following diagram. 
$$\xymatrix{
 & \widetilde{B(K\rtimes G)}_{\im(s_d)} \ar[d]^-{\pi_d} \\
P_{BG}B(K\rtimes G) \ar[r]_{\ev_1} \ar@{-->}[ru]^{\widetilde{\ev_1}} & B(K\rtimes G)
}$$
By construction $\im (\pi_d)_{_\sharp}=\im (\ev_1)_{_\sharp}$, so 
the lifting criterion for covering spaces implies that there exists a lifting $\widetilde{\ev_1}$
for $\ev_1$ along $\pi_d$ and that it induces isomorphism on the fundamental groups. Since both domain
and codomain are Eilenberg-MacLane spaces, $\widetilde{\ev_1}$ is a homotopy equivalence. As a consequence, 
every homotopy section of $\pi_d$ corresponds to a homotopy section of $\ev_1$ and vice-versa, 
therefore $\tc(d)=\secat(\ev_1)=\secat(\pi_d)$. 
\end{proof}

Sectional category of a covering map is 0 if, and only if, the covering is trivial, which immediately
implies the following characterization.

\begin{corollary}
$\tc(d)=0$ if, and only if $K$ is a trivial group.
\end{corollary}

\subsection{Group-theoretic lower bounds}\label{subs:Group-theoretic lower bounds}
In this section we adapt lower bounds for sectional category 
from \cite{GrantLuptonOprea2015} and \cite{EBFMO} to our situation. The following
lemma is essentially \cite[Cor. 2.3]{GrantLuptonOprea2015} but their statement is slightly different
and the proof is not self-contained, so we reprove the result here for the convenience of the reader.

\begin{lemma}\label{lem:secat of inclusion}
Let $H$ and $K$ be subgroups of a group $G$ such that for every $g\in G$ the conjugate 
subgroup $K^g=g K g^{-1}$ trivially intersects $H$. Then $\secat(\widetilde{BG}_K\to BG)\ge\cd(H)$.    
\end{lemma}
\begin{proof}
Form the following pull-back diagram, where $\pi_H$ and $\pi_K$ are the covering projections
corresponding respectively to the inclusions of $H$ and $K$ in $G$.
$$\xymatrix{
E \ar[r] \ar[d]_\pi \pullbackcorner & \widetilde{BG}_K\ar[d]^{\pi_K}\\
\widetilde{BG}_H\ar[r]_{\pi_H} & BG 
}$$
Then $\pi\colon E\to \widetilde{BG}_H=BH$ is also a covering projection and the path components 
of $E$ are coverings of 
$BH$ that correspond to subgroups $H\cap K^g$ for $g\in G$, which are all trivial by our assumptions.
We conclude that all components of $E$ are contractible, therefore $\secat(\pi)=\cat(BH)=\cd(H)$. 
Since sectional category decreases along pull-backs, we obtain our claim.
\end{proof}

In \cite{EBFMO} the authors proved the following lower bound for the sectional category of 
the covering projection determined by a subgroup inclusion: let $\Gamma$ be a geometrically finite group 
(in the sense that $B\Gamma$ has the homotopy type of a finite CW-complex) and let 
$\pi\colon BG\to B\Gamma$ be the covering projection induced by the inclusion of some subgroup 
$G\le\Gamma$. Then \cite[Thm. 5.2]{EBFMO} states that
$$\secat(\pi)\ge\cd(\Gamma)-\kappa, $$
where $\kappa=\max\{\cd(G\cap G^\gamma)\mid \gamma\in \Gamma-G\}.$

We are interested in the subgroups $s_d(G)\le  K\rtimes G$, where $s_d$ is a section determined by
a derivation $d\colon G\to K$. In general $\cd(\Gamma)\le\cd(G)+\cd(K)$, and the equality often holds, 
for example if $G$ and $K$ are duality groups (e.g., fundamental groups of closed aspherical manifolds).
We may also considerably simplify computation of $\kappa$ in the case of the standard 
inclusion of $G$ in $K\rtimes G$. First note
that $(k,g)\in K\rtimes G-G$ precisely when $k\ne 0$. Then
$$(k,g)\cdot G\cdot (k,g)^{-1}=(k,1)\cdot G\cdot (k,1)^{-1}=\{(k-g\cdot k,g)\mid k\in K\}.$$
The intersection of $G$ with its conjugate is 
$$G\cap (k,g)\cdot G\cdot (k,g)^{-1}=\{g\in G\mid k-g\cdot k=0\}=\mathrm{Stab}_G(k).$$
Therefore, in that case 
$$\kappa=\max\{\cd(Stab_G(k))\mid k\in K\}.$$ 
In particular, if $G$ fixes some non-zero element of $K$ (e.g., if $G$ acts trivially on $K$),
then $\kappa=\cd(G)$. At the opposite extreme, if $G$ acts freely on $K$, then $\kappa=0$ and,
since $\secat(\pi)\le\cd(G)$ we get $\secat(\pi)=\cd(G)$. 

\begin{example}
A simple example of free action is $\varphi\colon\ZZ\to \mathrm{GL}_n(\ZZ)=\Aut(\ZZ^n)$, $\varphi(k)=A^k$, 
where $A$ is a matrix with the property that none of its eigenvalues is a root of unity. 
In that case, we have $\cd(\ZZ^n\rtimes_A\ZZ)=n+1$ and for the trivial derivation 
$d_0\colon\ZZ\to\ZZ^n$ we obtain $\kappa=0$, therefore $\tc(d_0)=n+1$. Actually, since 
$\ZZ^n$ is abelian, we may apply Corollary \ref{cor:K abelian} to show that 
$\tc(d)=n+1$ for all derivations $d$.
\end{example}

For non-trivial derivations, the computation of $\kappa$ is usually more complicated. 
However, in the following section we will see 
that it is often sufficient to consider only the trivial derivation corresponding to the 
standard inclusion of $G$ in $K\rtimes G$.

\subsection{Invariance properties} In order to define the parametrized LS-category of a map,
it is necessary to choose a specific section. This is why we had to specify both an action 
$\varphi\colon G\to K$ and a derivation $d\colon G\to K$. However, in the majority of our computations, 
all derivations relative to a fixed action give the same value, so in these cases $\tc$ is 
actually dependent
only on the action $\varphi$. On the other hand, as shown in Example \ref{ex:F2 identity} this is not
true in general. In this section we are going to show that the dependence on the choice 
of a specific derivation is related to the existence of automorphisms of $K\rtimes G$ that fix $G$. 
We remind the reader that, despite additive notation, the group $K$ is not assumed to
be commutative.

\begin{lemma}\label{lem:tc(d)=tc(d')}
 Let $d,d'\colon G\to K$ be derivations relative to some action of $G$ on $K$. If there exists 
 an automorphism $\varphi\colon K\rtimes G\to K\rtimes G$ such that $\varphi\circ s_d=s_{d'}$, 
 then $\tc(d)=\tc(d')$. 
\end{lemma}
\begin{proof}
By assumption, there is a commutative diagram
$$\xymatrix{
& G \ar[dl]_{s_d}  \ar[dr]^{s_{d'}}\\
K\rtimes G \ar[rr]_{\varphi} & & K\rtimes G
}$$
From this we obtain a homotopy commutative diagram of spaces 
$$\xymatrix{
& {\widetilde{B(K\rtimes G)}_{\im(s_d)}} \ar[dl]_-{\pi_d} \ar[dr]^-{\pi_{d'}}\\
B(K\rtimes G) \ar[rr]_{B\varphi} & & B(K\rtimes G)
}$$
Since $B\varphi$ is a homotopy equivalence, it follows that $\secat(\pi_d)=\secat(\pi_{d'})$, which by
Theorem \ref{thm:tc=secat}(2) implies $\tc(d)=\tc(d')$.
\end{proof}

Let $\varphi$ be an automorphism of $K\rtimes G$, which fixes $G$ in the sense that the following diagram
commutes:
$$\xymatrix{
K\rtimes G \ar[rr]^\varphi \ar[rd]_p & & K\rtimes G \ar[ld]^p\\
& G
}$$
Then for every section $s$ of $p$, $(p\circ\varphi)\circ s=p\circ s=1_G$ implies that
$\varphi\circ s$ is also a section $s$. In view of the above lemma it makes sense to
determine all automorphisms $\varphi$ that fix $G$. Any such automorphism must satisfy
$$\varphi(k,g)=\varphi(k,1)\cdot \varphi(0,g).$$
Since $\varphi(k,1)$ depends only on $k\in K$ it must be of the form 
$\varphi(k,1)=(\alpha(k),1)$ for some homomorphism $\alpha\colon K\to K$. 
Actually, since $\varphi$ sends $K$ to itself, $\alpha$ must be an automorphism.

Similarly, $p\circ \varphi(0,g)=g$ implies that $g\mapsto \varphi(0,g)$ is a section of 
$p$, so it must be of the form $\varphi(0,g)=(\beta(g),g)$ where $\beta\colon G\to K$
is a derivation. We conclude that 
$$\varphi(k,g)=(\alpha(k),1)\cdot (\beta(g),g)=(\alpha(k)+\beta(g),g)$$
for some automorphism $\alpha$ and some derivation $d$. Obviously, the choice of $\alpha$ and $d$ 
is not arbitrary.

\begin{lemma}\label{lem:phi=alpha+beta}
$\varphi\colon K\rtimes G\to K\rtimes G$ is an automorphism that fixes $G$ if, and only if, 
$\varphi(k,g)=(\alpha(k)+\beta(g),g)$, where $\alpha\colon K\to K$ is an automorphism, 
$d\colon G\to K$ is a derivation and the relation $\alpha(g\cdot k)+\beta(g)=\beta(g)+g\cdot\alpha(k)$
holds for all $(k,g)\in K\rtimes G$.
\end{lemma}
\begin{proof}
Given $\alpha$ and $\beta$ as in the statement define $\varphi(k,g):=(\alpha(k)+\beta(g),g)$ and check
whether it defines a homomorphism. We have
$$\varphi((k,g)\cdot(k',g'))=\varphi(k+g\cdot k',g\cdot g')=
(\alpha(k)+\alpha(g\cdot k')+\beta(g)+g\cdot \beta(g'),g\cdot g')$$
and 
$$\varphi(k,g)\cdot\varphi(k',g')=(\alpha(k)+\beta(g),g)\cdot(\alpha(k')+\beta(g'),g')=
(\alpha(k)+\beta(g)+g\cdot \alpha(k')+g\cdot \beta(g'),g\cdot g').
$$
Comparison of the two expressions shows that, $\varphi$ is a homomorphism if, and only if, 
$\alpha(g\cdot k)+\beta(g)=\beta(g)+g\cdot\alpha(k)$ for all $(k,g)\in K\rtimes G$. Actually, 
$\varphi$ is an automorphism, since its inverse is given by 
$\varphi^{-1}(k,g)=(\alpha^{-1}(k)- \alpha^{-1}(\beta(g)),g).$
\end{proof}

We are now ready to prove the main result of this section.

\begin{theorem}\label{thm:tc(d)=tc(ad+b)}
Let $\alpha\colon K\to K$ be an automorphism, and let $d\colon G\to K$ be a derivation such that 
$\alpha(g\cdot k)+\beta(g)=\beta(g)+g\cdot\alpha(k)$ holds for all $(k,g)\in K\rtimes G$.   
Then for every derivation $d\colon G\to K$ $\alpha\circ d+\beta$ is also a derivation and 
$\tc(\alpha\circ d+\beta)=\tc(d)$.
\end{theorem}
\begin{proof}
By Lemma \ref{lem:phi=alpha+beta}  the formula $\varphi(k,g):=(\alpha(k)+\beta(g),g)$ defines 
an automorphism of $K\rtimes G$ and we have
$$\varphi(s_d(g))=\varphi(d(g),g)=(\alpha(d(g))+\beta(g),g).$$
This shows that $\alpha\circ d+\beta$ is also a derivation, while Lemma \ref{lem:tc(d)=tc(d')} implies that 
$\tc(\alpha\circ d+\beta)=\tc(d)$.
\end{proof}

We have two main types of automorphisms that satisfy the assumptions of the above theorem, which
we call \emph{conjugate} and \emph{central}.

\begin{example}\label{ex:conjugate}
Let $a\in K$ and define 
$$\alpha(k):=a+k-a\ \ \text{and}\ \ \beta(g)=a-g\cdot a.$$ 
It is easy to check that 
$\alpha$ and $\beta$ satisfy the assumptions of Lemma \ref{lem:phi=alpha+beta}. 
By Theorem \ref{thm:tc(d)=tc(ad+b)}, for every derivation $d\colon G\to K$ the formula
$d'(g)=a+d(g)-g\cdot a$ also defines a derivation, and $\tc(d)=\tc(d').$

This result has a very nice formulation in terms of $K$-torsors. We refer the reader to 
\cite[Sec. I.2]{NSW} for relevant definitions and results. Every derivation $d\colon G\to K$
determines a $K$-torsor $X_d$ in the category of $G$-sets, and two $K$-torsors $X_d$ and 
$X_{d'}$ are isomorphic precisely when $d'(g)=a+d(g)-g\cdot a$ for some $a\in K$.
Thus we may state the following consequence of Theorem \ref{thm:tc(d)=tc(ad+b)}:

\begin{corollary}\label{cor:torsors}
If $d$ and $d'$ correspond to isomorphic torsors 
(i.e., elements of $H^1(G;K)$), then $\tc(d)=\tc(d')$.    
\end{corollary}

If $K$ is abelian, thus a $G$-module, then isomorphism classes of $K$-torsors form a group, which can be
naturally identified with $H^1(G;K)$, the first cohomology of the group $G$ with coefficients in $K$. 
Thus torsors represent a natural way to extend group cohomology to non-abelian coefficients, at least
in low dimensions. We will exploit this point of view in the next example.\qed
\end{example}

\begin{example}\label{ex:central}
Let $b\in Z(G)$  and let $c\colon G\to K$ be a derivation with $c(G)\subseteq Z(K)$. Define
$$\alpha(k):=b\cdot k\ \ \text{and}\ \ \beta(g):=c(g),$$
and check that the assumptions of Lemma \ref{lem:phi=alpha+beta} are satisfied. Then for every
derivation $d\colon G\to K$ the formula $d'(g)=b\cdot d(g)+c(g)$ also defines a derivation, 
and $\tc(d)=\tc(d').$ In other words, $\tc$ is invariant with respect to dilation and translation
of derivations by central elements. The following special case is of particular interest:

\begin{corollary}\label{cor:K abelian}
Assume $K$ is abelian. Then for any two derivations $d,d'\colon G\to K$ we have $\tc(d)=\tc(d')$. 
In other words, if $K$ is abelian, then $\tc$ is an invariant of action alone.
\end{corollary}
\qed
\end{example} 

We may summarize the conclusions of the above examples by considering 
the exact sequence of cohomology groups and pointed sets of $G$ associated to the short 
exact sequence of coefficients 
$Z(K)\stackrel{i}{\rightarrowtail} K\stackrel{q}{\twoheadrightarrow} \mathrm{Inn}(K)$
(see \cite[Ch. III]{Giraud1971})
$$
\xymatrix{
H^1(G;Z(K))\ar[r]^-{i^*} & H^1(G;K) \ar[r]^-{q^*} & H^1(G;\mathrm{Inn}(K))\ar[r]^\delta & H^2(G;Z(K))
}$$
The map $\delta$ is often called \emph{Brauer obstruction}.
By Example \ref{ex:conjugate} the invariant $\tc$ is defined on elements of $H^1(G;K)$. Furthermore, by
Example \ref{ex:central} it is constant on the image of $i^*$, so $\tc$ is, in fact, defined on 
$\im(q^*)=\Ker\delta$. Finally, the action of $G$ on $\mathrm{Inn}(K)$ induces and action of 
$Z(G)\le G$ on $\mathrm{Inn}(K)$, and therefore on $H^1(G;\mathrm{Inn}(K))$. We have seen in Example 
\ref{ex:central} that $\tc$ is constant on the orbits of that action. We have thus proved the following 
result.

\begin{corollary}
Let $d,d'\colon G\to K$ be derivations viewed as elements of $H^1(G;K)$, and let 
$\delta\colon H^1(G;\mathrm{Inn}(K))\to H^2(G;Z(K))$ be the Brauer obstruction map. 
If $d$ and $d'$ map to the same class in $\Ker\delta/Z(G)$, then $\tc(d)=\tc(d')$.
\end{corollary}

\subsection{Compositions of splittings} In this section, we describe estimates for the complexity of
iterated semi-direct products. We begin with a general result on sectional category. 

\begin{proposition}\label{prop:secat of composition}
Let $f\colon X\to Y$ and $g\colon Y\to Z$ be any maps. Then 
$$\secat(g\circ f)\ge\secat(g).$$
Furthermore, if $g$ admits a retraction, that is, a map $r\colon Z\to Y$, such that 
$r\circ g\simeq 1_Y$, then 
$$\secat(g\circ f)\ge\secat(f).$$
On the other hand, if $g$ admits a section, that is, a map $s\colon Z\to Y$, such that $g\circ s\simeq 1_Z$, then 
$$\secat(g\circ f)\le\secat(f).$$
\end{proposition}
\begin{proof} The inequality $\secat(g \circ f)\geq\secat(g)$ is immediate and is left to the reader. 

Toward the second estimate, assume that an open subset $U\subseteq Z$ admits a homotopy
section $\sigma\colon U\to X$, such that $(g\circ f)\circ \sigma\simeq i_U\colon U\hookrightarrow Z$.
Let $V:=g^{-1}(U)\subseteq Y$ and let $i_V\colon V\to Y$ be the inclusion. 
Then we may use the relation $r\circ g\simeq 1_Y$ to compute 
$$f\circ (\sigma\circ g)\simeq r\circ (g\circ f\circ \sigma)\circ g \simeq r\circ i_U\circ g
=(r\circ g)\circ i_V\simeq i_V$$
and thus conclude that $\sigma\circ g$ is a section of $f$ over $V$. Since the preimages of 
sectionable subsets of $Z$ cover $Y$, we conclude that $\secat(g\circ f)\ge\secat(f)$.

For the third estimate, assume that an open subset $U\subseteq Y$ admits a homotopy section 
$\sigma\colon U\to X$, such that $f\circ \sigma\simeq i_U\colon U\hookrightarrow Y$.

Let $V:=s^{-1}(U)\subseteq Z$ and let $i_V\colon V\to Z$ be the inclusion. 
Then we may use the relation $g\circ s\simeq 1_Z$ to compute 
$$(g\circ f)\circ (\sigma\circ s)\simeq g\circ i_U\circ s=g\circ s\circ i_V\simeq i_V$$
and thus conclude that $\sigma\circ s$ is a section of $g\circ f$ over $V$. Since the preimages of 
sectionable subsets of $Y$ cover $Z$, we conclude that $\secat(f)\ge \secat(g\circ f)$.
\end{proof}

Returning to semi-direct products, consider groups $G,H,K$ with an action of $G$ on $H$ and an 
action of $H\rtimes G$ on $K$, so that 
we can form the iterated semi-direct product $K\rtimes(H\rtimes G)$. Furthermore, we have projections
$$p\colon K\rtimes(H\rtimes G)\to H\rtimes G\ \ \ \text{and}\ \ \ q\colon H\rtimes G\to G.$$
Note that in $K\rtimes(H\rtimes G)$ the subgroup $G$ acts both on $H$ and $K$, resulting in
an action of $G$ on $K\rtimes H$. 
A direct computation shows that the kernel of $q\circ p\colon K\rtimes(H\rtimes G)\to G$ is 
isomorphic to $K\rtimes H$, so that $K\rtimes(H\rtimes G)\cong (K\rtimes H)\rtimes G$. (Here, we do not
claim that the semi-direct
product operation is associative. In fact, given an action of $H$ on $K$
and an action of $G$ on $K\rtimes H$, one can form $(K\rtimes H)\rtimes G$, but this does not imply
that $G$ acts on $K$ as in the case of $K\rtimes(H\rtimes G)$.)

Let us now pick a derivation $d\colon G\to H$ and a derivation $d'\colon H\rtimes G\to K$. We may combine 
them to obtain a derivation 
$$d'\ast d:=(d'\circ s_d,d)\colon G\to K\rtimes L$$
so that $s_{d'\ast d}=s_{d'}\circ s_d$. 
If we apply the classifying space functor to the diagram 
$$\xymatrix{
K\rtimes(H\rtimes G) \ar@<-.5ex>[r]_-p  & H\rtimes G \ar@<-.5ex>[r]_-q 
\ar@<-.5ex>[l]_-{s_{d'}} &  G \ar@<-.5ex>[l]_-{s_d}
}$$
and use Theorem \ref{thm:tc=secat} and Proposition \ref{prop:secat of composition}, we obtain
the following result:

\begin{theorem}
Let $d\colon G\to H$ be a derivation relative to some action of $G$ on $H$, and let 
$d'\colon H\rtimes G\to K$ be a derivation relative to some action of 
$H\rtimes G$ on $K$. Then the complexity of the derivation $d'\ast d=(d'\circ s_d,d)$
satisfies the inequality
    $$\tc(d'\ast d)\ge\max\{\tc(d),\tc(d')\}$$
\end{theorem}

At this point, we would like to mention a question whether there is an upper estimate of
$\tc(d'\ast d)$ in terms of $\tc(d)$ and $\tc(d')$, perhaps by their sum. 
In terms of our notation in Proposition \ref{prop:secat of composition}, 
Arkowitz and Strom \cite[Theorem 5.4]{ArkowitzStrom2004} proved that 
$\secat(s\circ s')\le \cat_{FH}(s)+\secat(s')$, where $\cat_{FH}(s)$ denotes a version
of category of a map introduced by Fadell and Husseini (the latter can be also expressed
in terms of the relative category of the pair $(E,A)$, see \cite[Section 7.2]{CLOT}). 
Unfortunately, we do not know the precise relation between $\cat_{FH}(\pi_{d'})$ and
$\tc(d')$. 





\section{Examples and computations}\label{sec:Examples and computations}

In this section we collect a series of computations that cover some important semi-direct
decompositions of torsion-free groups. In order to avoid repetition we selected examples 
that illustrate 
different approaches and general methods developed in Section \ref{sec:Topological complexity 
of a derivation}. We first consider fundamental groups of mapping tori.

\subsection{Mapping tori}\label{ssec:mapping tori}

Let $f\colon X\to X$ be a self-map of a topological space $X$.  The \emph{mapping torus} of $f$ 
is the quotient space 
$$T_f=(X\times[0,1])/(x,1)\sim (f(x),0)$$
Thus, $T_f$ is obtained from the cylinder $X\times [0,1]$ by gluing the
top copy of $X$ to the bottom copy using the map $f$. The mapping torus comes 
with the obvious projection map $p\colon T_f\to S^1$. If $f$ 
is a homeomorphism, then $p\colon T_f\to S^1$ is a fiber bundle with fibre $X$.

Assume that $X$ is path-connected and that $f$ is a homotopy equivalence, so that 
the induced homomorphism $f_\sharp\colon \pi_1(X,x_0)\to \pi_1(X,f(x_0))$ is an isomorphism.
In order to obtain an automorphism of $\pi_1(X,x_0)$ choose a path $\alpha\colon (I,0,1)\to (X,f(x_0),x_0)$,
denote by $\widetilde{\alpha}\colon\pi_1(X,f(x_0))\to \pi_1(X,x_0)$ the change-of-basepoint isomorphism
along $\alpha$, and then define $\varphi:=\widetilde\alpha\circ f_\sharp$. With this notation, the fundamental
group of $T_f$ is a semi-direct product of the form
$$\pi_1(T_f;(c_0,0)\cong \pi_1(X,x_0)\rtimes_{\varphi}\mathbb Z.$$
Taking a different path from $f(x_0)$ to $x_0$, the automorphism $\varphi$ changes by an inner
automorphism of $\pi_1(X,x:0)$, so it only affects the representation of the group as a semi-direct 
product. Moreover, the action of an inner automorphism on a $d$ does not change the corresponding
class in $H^1(\ZZ; \pi_1(X,x_0)$, so by Corollary \ref{cor:torsors} it also does not change the 
value of $\tc(d)$. From now on, we will drop the basepoints from our notation for the fundamental
group. 

By Theorem \ref{thm:cd estimates for tc}, for every derivation $d\colon\ZZ\to\pi_1(X)$
$$\cd(\pi_1(X)\le \tc(d)\le\cd(\pi_1(X))+1,$$
so we only need to decide between two possible values.

\begin{example} (The Klein bottle) \label{ex:Klein bottle}
Let us denote by $K$ the Klein bottle which can be  obtained 
as the mapping torus of the 
reflection map on $S^1$. Its fundamental group is therefore a semi-direct product, determined by 
the action $\varphi\colon\ZZ\to \Aut(\ZZ)=\{\pm 1\}$, where $\varphi(1)$ is the multiplication by $-1$.
The standard presentation of $\pi_1(X)$ is
$$\pi_1(K)=\ZZ\rtimes_\varphi \ZZ=\bigl\langle\, x,t \;\bigm|\; txt^{-1}=x^{-1} \,\bigr\rangle$$    

\begin{proposition}
In the Klein bottle group $\tc(d)=2$ for every derivation $d$.   
\end{proposition}
\begin{proof}
Since $\ZZ$ is abelian, the value of $\tc$ depends only on the action, and it is sufficient to
compute $\tc(d_0)$ for the trivial derivation $d_0$.
We know that $1\le\tc(d_0)\le 2$ and we will use Theorem \ref{thm:coho estimate for tc} to show that 
the correct value is 2. The mod-2 cohomology ring of the Klein bottle is well-known and is given as
$$ H^*(K;\mathbb{F}_2)\;=\;\mathbb{F}_2[a,b]\big/\bigl(a^2,\;b^2+ab\bigr),   \qquad |a|=|b|=1. $$

The trivial section associated to $d_0$ is $s_0\colon S^1\to K$, and it is easy to see that 
$s_0^\ast(a)$ is the generator of $H^1(S^1,\mathbb{F}_2)$ and $s_0^\ast(b)=0$. Thus
$\Ker(s_0^\ast)$ is the linear span of $b$ and  $ab$. It immediately follows that $b^2=ab\ne 0$, 
so by Theorem \ref{thm:coho estimate for tc}
$$\tc(d_0)\ge\cl(\Ker s_0^\ast)=2.$$
\end{proof}
\end{example}

The reader may compare the last result with the case of the two-dimensional torus, whose fundamental group 
is given by the trivial 
action of $\ZZ$ on $\ZZ$. Then $\tc(d)=1$ as shown in Example \ref{ex:trivial derivation}.

\begin{example} ($F_2\times \ZZ$.) \label{ex:trivial F2}
Let us now consider the trivial action of $\ZZ$ on 
the free group $F_2$. Since the latter is not abelian, it is possible that the values of $\tc(d)$ 
depend on the choice of derivation. For the trivial derivation $d_0$ Example \ref{ex:trivial derivation}
gives $\tc(d_0)=\cd(F_2)=1$. Since the action is trivial, derivations correspond to homomorphisms
$d\colon \ZZ\to F_2$, so given a word $w\in F_2$, we denote by $d_w$ the derivation determined 
by $d_w(1)=w$.

\begin{proposition}\label{prop:ex free by Z}
$\tc(d_w)=2$ for every $w\ne 1$.
\end{proposition}
\begin{proof}
We will base our proof on the lower estimate given in Lemma \ref{lem:secat of inclusion}.
Let $F_2$ be the free group with generators $x_1,x_2$ and let 
$\epsilon_1,\epsilon_2\colon F_2\to\ZZ$ be the homomorphisms that to each word 
$w\in F_2$ assigns the sum of exponents of $x_1$, and respectively $x_2$ in $w$.
Furthermore, let $G:=F_2\times\ZZ$, $K:=\im(s_{d_w})=\{(w^k,k)\in F_2\times\ZZ\}$ and 
$H:=\langle g\rangle \times \ZZ\le F_2\times \ZZ$, where 
$$g:=\left\{\begin{array}{ll}
x_1  & \epsilon_2(w)\ne 0 \\
x_2  & \epsilon_2(w)=0
\end{array}\right.$$
We claim that the groups $G,K,H$ satisfy the hypothesis of Lemma~\ref{lem:secat of inclusion}. 
Indeed, for elements $(u,l)\in G$ and $(w^k,k)\in K$ we have
$$(u,l)(w^k,k)(u,l)^{-1}=(uw^ku^{-1},k).$$
This is an element of $H$ if, and only if, $uw^ku^{-1}=g^l$ for some $l\in\ZZ$. 
We must consider two possibilities.

If $\epsilon_2(w)=0$, then $g=x_2$ and 
$0=\epsilon_2(uw^ku^{-1})=\epsilon_2(g^l)=l$ implies $uw^k u^{-1}=1$. In the free group
this is possible only if $w=1$, contrary to our assumption. 

If $\epsilon_2(w)\ne 0$, then $g=x_1$ and 
$k\,\epsilon_2(w)=\epsilon_2(uw^ku^{-1})=\epsilon_2(g^l)=0$ implies $k=0$. 
It follows that $l=\epsilon_1(g^l)=0$, 
therefore $H$ and the conjugate of $K$ intersect trivially. 

Finally, by Lemma~\ref{lem:secat of inclusion}, $\tc(d_w)\ge\cd(H)=2$. 
On the other hand $\tc(d_w)\le\cd(F_2\times \ZZ)=2$, therefore the equality holds.
\end{proof}
\end{example}

\begin{example} (Linear mapping tori) In this section we consider semi-direct products of the form
$\ZZ^n\rtimes \ZZ$ in which the action $\varphi\colon\ZZ\to\Aut(\ZZ^n)=\mathrm{Gl}_n(\ZZ)$ is determined
by a choice of an invertible integer matrix $\varphi(1)=A\in \mathrm{Gl}_n(\ZZ)$. By
Theorem \ref{thm:cd estimates for tc}, $\tc(d)$ is either $n$ or $n+1$, and by Corollary 
\ref{cor:K abelian} all derivations give the same value, so it is sufficient to compute $\tc(d_0)$
for the trivial derivation $d_0$. We are going to use \cite[Thm. 5.2]{EBFMO} 
that we already discussed in Section \ref{subs:Group-theoretic lower bounds} and which states that
$$\secat(BG\to B\Gamma)\ge\cd(\Gamma)-\kappa, $$
where $\kappa=\max\{\cd(G\cap G^\gamma)\mid \gamma\in \Gamma-G\}.$ 
By our discussion at the end of section \ref{subs:Group-theoretic lower bounds}, 
$\kappa$ is the maximum over all non-zero elements $x\in\ZZ^n$ of the stabilizers of the action of 
$\ZZ$ on $x$. The action of $k\in\ZZ$ on $x\in\ZZ^n$ is given by $k\cdot x=A^kx$, so we obtain 
the following result.

\begin{proposition}
Let $A\in\mathrm{Gl}_n(\ZZ)$ be a matrix that does not have roots of unity in its
spectrum. 
Then for every derivation $d\colon\ZZ\to\ZZ^n$ in $\ZZ^n\rtimes_A\ZZ$ we have $\tc(d)=n+1$.
\end{proposition}

The situation is more complicated when the spectrum of $A$ contains roots of unity. We are able to treat
some special cases, analogous to the Klein bottle from Example \ref{ex:Klein bottle}. For simplicity,
let us consider one specific instance, the semi-direct product 
$$\ZZ^3\rtimes_A \ZZ,$$ 
where $A$ is a diagonal matrix with diagonal
entries 1,-1 and -1. The $\mathbb{F}_2$-cohomology of $\ZZ^3\rtimes_A \ZZ$ can be computed 
using Bockstein
operations: it has four generators, $u$ which comes from the projection of the mapping torus 
to the base circle,
and $x,y,z$, coming from the fibre. Their products satisfy the relations
$$u^2=x^2=0,\ \ y^2=uy  \ \ z^2=uz\ \ \text{and}\ \ uxyz\ne 0.$$
The classes $x,y,z$ are in the kernel $\Ker s^\ast$ of the canonical section $s$. Since
$xy^2z=uxyz$, the cup length of $\Ker s^\ast$ is 4, therefore $\tc(d)=4$ for every 
derivation $d\colon\ZZ\to \ZZ^3$.

To contrast, consider the integral Heisenberg group 
\[
H=\left\{
\begin{pmatrix}
1 & a & c\\
0 & 1 & b\\
0 & 0 & 1
\end{pmatrix}
\;\middle|\;
a,b,c\in\mathbb Z
\right\}.
\]

Let $K\le H$ consist of matrices with $b=0$, and let $G\le H$ consist of matrices with $a=c=0$.
Then $K\cong\ZZ^2$, $G\cong \ZZ$ and $H$ is isomorphic to the semi-direct product $K\rtimes G$, where 
the action is determined by the matrix 
$$
A=
\begin{pmatrix}
1 & 1\\
0 & 1
\end{pmatrix}
$$
The cohomology ring of $BH$ with $\Z_2$-coefficients can be computed from the 
Lyndon-Hochschild-Serre spectral sequence, which yields additive generators
$a,b\in H^1(H;\Z_2), u,v\in H^2(H;\Z_2)$ and $w\in H^3(H;\Z_2)$, with the relations
$a^2=b^2=ab=0$, $au=bv=0$ and $av=bu=w$ (other products are zero for dimensional reasons).
Furthermore, section $s\colon G\to H$ satisfies $s^*(a)=0, s^*(b)\ne 0$ in $H^1(G;\Z_2)$, so
$\Ker(s^*)$ is the linear span of elements $a,u,v,w$. Clearly, the cup-length 
of $\Ker(s^*)$ is 2, therefore $\tc(d)\ge 2$. On the other hand, the dimensional upper 
bound is $\tc(d)\le \cd(H)=3$, which leaves us with a partial result $\tc(d)\in\{2,3\}$ for
derivations of the Heisenberg group.

\qed
\end{example}

\subsection{Pure braid groups} One of our motivating examples for the study of semi-direct products 
was the standard decomposition of pure braid groups (see \cite{KasselTuraev2008} for the notation and main
properties). Let $P_n$ denote the pure braid group on $n$-strands. By erasing the $n$-th strand we
obtain a projection $p\colon P_n\to P_{n-1}$, whose kernel is isomorphic to $F_{n-1}$ , the free
group on $n-1$ generators. The obvious inclusion of $s\colon P_{n-1}\to P_n$ is
a section of $p$, so we obtain a decomposition $$P_n\cong F_{n-1}\rtimes P_{n-1}.$$
The action of $P_{n-1}$ on $F_{n-1}$ can be described explicitly in terms of standard generators, 
see \cite[Sec.1.3]{KasselTuraev2008}. Note that the inclusion  $s_0\colon P_{n-1}\to P_n$ mentioned above
corresponds to the trivial derivation $d_0\colon P_{n-1}\to F_{n-1}$. 

Let us next recall a very useful model for the classifying space of a pure braid group. Denote by 
$F(\RR^2,n)$ the configuration space of $n$ different points in the plane. It is well known 
(cf. \cite[Sec. 1.4]{KasselTuraev2008}) that $F(\RR^2,n)$ is aspherical and that its fundamental group 
is precisely $P_n$, therefore $F(\RR^2,n)\simeq BP_n$. In addition, the \emph{Fadell-Neuwirth fibration}
$q\colon F(\RR^2,n)\to F(\RR^2,n-1)$, given by the projection to the first $n-1$ components, induces the
homomorphism $p\colon P_n\to P_{n-1}$, and the inclusion $i\colon F(\RR^2,n-1)\to F(\RR^2,n)$ induces 
the section $s\colon P_{n-1}\to P_n$. Finally, the fibre of $q$ over a point 
$(x_1,\ldots,x_{n-1})\in F(\RR^2,n-1)$ is clearly $\RR^2-\{x_1,\ldots,x_{n-1}\}$, which is homotopy 
equivalent to the wedge of $n-1$ circles, which we denote $V_{n-1}$. To summarize, 
the fibration sequence 
$$\xymatrix{
V_{n-1}\ar@{^(->}[r] & F(\RR^2,n) \ar[r]_-q &  F(\RR^2,n-1)) \ar@<-1ex>[l]_-i
}$$
models the classifying space fibration associated to the split short exact of groups
$$\xymatrix{
F_{n-1}\ \ar[r] & P_n \ar[r]_-p & P_{n-1} \ar@<-1ex>[l]_-{s_0}
}$$
We may conclude that the topological complexity of the trivial derivation $d_0\colon P_{n-1}\to F_{n-1}$ 
can be computed as
$$\tc(d_0)=\cat_{F(\RR^2,n-1)}(F(\RR^2,n))=\secat(i\colon F(\RR^2,n-1)\to F(\RR^2,n)).$$

More generally,  we may iterate the projections between pure braid groups. 
Then, for every $2\le m<n$, we obtain a projection $p\colon P_n\to P_m$ which admits a section
$s_0\colon P_m\to P_n$. Therefore, there is an action of $P_m$ on $\Ker(p)$, $P_n\cong \Ker(p)\rtimes P_m$ 
and the section $s_0$ corresponds to the trivial derivation $d_0\colon P_m\to\Ker(p)$.
Clearly, $\tc(d_0)=\secat(i\colon F(\RR_m,2)\to F(\RR_n,2))$.

\begin{theorem}
Consider the semi-direct product decomposition $P_n\cong \Ker(p)\rtimes P_m$ for $2\le m<n$, and let
$d_0\colon P_m\to \Ker(p)$ be the trivial derivation. Then $\tc(d_0)=n-1$.    
\end{theorem}
\begin{proof}
We need to compute $\secat(i\colon F(\RR_m,2)\to F(\RR_n,2))$. By Theorems \ref{thm:catB(X)}(3)
and \ref{thm:coho estimate for tc}
$$\cat(F(\RR_n,2))\ge \secat(i\colon F(\RR_m,2)\to F(\RR_n,2))\ge \cl(\Ker(i^*)).$$
It is known that the homotopy dimension of $F(\RR_n,2)$ is $n-1$, therefore 
$\cat(F(\RR_n,2))\le n-1$. 

For the lower bound, we use the cohomological estimate based on the description and properties 
of the cohomology ring of the configuration spaces as given in \cite[Ch. V]{FadellHusseini2001}.
The integral cohomology of $F(\R^{2},n)$ is commutative graded algebra
$$  H^{*}\!\bigl(F(\R^{2},n);\Z\bigr)
  =\Lambda(\omega_{i,j})\big/
   \bigl(\omega_{i,j}\omega_{i,k}-\omega_{i,j}\omega_{j,k}+\omega_{i,k}\omega_{j,k}\bigr),
  \qquad |\omega_{ij}|=1,\ \ 1\le i<j\le n .
$$
In particular, the squarefree monomial 
$$\omega_{1,n}\cdot\omega_{2,n}\cdot\cdots\cdot\omega_{n-1,n}$$
is nonzero and it generates $H^{n-1}\bigl(F(\R^{2},n);\Z\bigr)$.
By comparing the cohomology rings of $F(\R^{2},n)$ and $F(\R^{2},n-1)$ we can easily deduce 
that elements $\omega_{1,n},\omega_{2,n}\ldots\omega_{n-1,n}$ are contained in $\Ker(i^*)$, 
therefore $\cl(\Ker(i^*))\ge n-1$. We have thus proved that $\secat(i)=n-1$.

The general claim follows by applying \ref{prop:secat of composition} to the sequence
$F(\RR^2,m)\hookrightarrow F(\RR^2,n-1)\hookrightarrow F(\RR^2,n)$.
 \end{proof}
In Example \ref{ex:trivial F2} we computed $\tc(d)$ for all derivations $\ZZ\to F_2$, while
for non-trivial action we usually computed $\tc(d)$ only for trivial derivation. In favorable circumstances, 
we can extend the computation for all derivation even when the action is non-trivial. One such example 
arise in pure braid groups.

\begin{example}
Consider the homomorphism $p\colon P_n\to P_2\cong\ZZ$, given by the projection of a 
braid onto the last two strands. 
This yields the decomposition of
$P_n$ as a semi-direct product $\Ker(p)\rtimes\ZZ$. However, we may change the set of generators for $P_n$ 
by replacing one of them with the element that represents the full twist of all the strands and  
generates the centre of $P_n$. This leads to the alternative representation $P_n\cong\Ker(p)\times\ZZ$.  
To simplify the notation, let us consider the case $P_3\cong F_2\rtimes \ZZ\cong F_2\times \ZZ$.
It is not difficult to compute the isomorphism $\theta\colon F_2\times \ZZ\to F_2\rtimes \ZZ$ explicitly.
Let $F_2$ be the free group on two letters $x,y$ and let $c=xy\in F_2$. The action $\varphi\colon\ZZ\to\Aut(F_2)$
that gives the semi-direct decomposition of $P_3$ sends the generator of $\ZZ$ to the conjugation by $c$. Thus, 
for every $w\in F_2$ we have 
$$[\varphi(k)](w)=c^k\,w\,c^{-k}.$$
It is easy to check that the isomorphism $\theta\colon F_2\times \ZZ\to F_2\rtimes \ZZ$ is given by 
the formula $$\theta(w,k):=(w\,c^k,k).$$
The crucial fact is that the isomorphism preserves the second coordinate, because it implies that for
every section $s\colon\ZZ\to F_2\times\ZZ$ we obtain the section $\theta\colon \ZZ\to F_2\rtimes \ZZ$.  
As a consequence, we get a bijection between derivations $\ZZ\to F_2$ with respect to the trivial action
and derivations $\ZZ\to F_2$ with respect to the Artin action $\varphi$. Explicitly, derivation
$d_w$ for the trivial action (notation as in  \ref{ex:trivial F2}) corresponds to derivation
$\tilde d_w\colon k\mapsto w\,c^k$ with respect to the action $\varphi$ (that $\tilde d_w$ is indeed
a derivation is easily checked by taking into account that $c$ is a central element). In particular,
to the trivial derivation $d_0$ corresponds the derivation $\tilde d_0\colon k\mapsto c^k$. By applying 
the classifying space functor we obtain the following diagram
$$\xymatrix{
& B\ZZ \ar[dl]_{\pi_w}  \ar[dr]^{\tilde \pi_w}\\
B(F_2\times\ZZ) \ar[rr]_{B\varphi} & & B(F_2\rtimes\ZZ)
}$$
where $\pi_w$ and $\tilde \pi_w$ are determined respectively by derivations $d_w$ and $\tilde d_w$.
Since $B\varphi$ is a homotopy equivalence, we obtain $\secat(\pi_w)=\secat(\tilde\pi_w)$. Finally, 
Theorem \ref{thm:tc=secat} implies the following result:

\begin{proposition}
Consider the semi-direct product decomposition $P_3=F_2\rtimes\ZZ$ and let $c$ be the product of the
generators of $F_2$ that generates the centre of $P_3$. For every $w\in F_2$ define 
a derivation $d_w\colon \ZZ\to F_2$, $d_2\colon k\mapsto w\, c^k$. Then 
$\tc(d_w)=1$ if $w=1$ and $\tc(d_w)=2$ if $w\ne 1$.
\end{proposition}

In a similar manner, we may compute $\tc$ for all derivations relative to the 
semi-direct decomposition of $P_n$ over $P_2$. Details are left to the reader.
\end{example}

\subsection{Almost-direct products of free groups} By iterating the semi-direct decomposition of
pure braid groups we may obtain a representation
$$P_n\cong F_{n-1}\rtimes(F_{n-2}\rtimes\cdots F_1)$$
as an iterated semi-direct product of free groups. D. Cohen \cite{Cohen} observed that many 
properties of pure braid groups may be extended to iterated semi-direct products of free groups
that satisfy an additional assumption. 

A group $\Gamma$ is an \emph{almost-direct product of free groups} if
$$\Gamma \;=\; \bigrt_{q=1}^{l} F_{d_q}   \;=\; F_{d_l}\rt\bigl(F_{d_{l-1}}\rt\cdots\rt F_{d_1}\bigr), $$
where $F_{d_q}=\langle x_{q,1},\dots,x_{q,d_q}\rangle$ is a free group of rank $d_q$ ($d_q\ge 1$) and such that
for each $1\le j<k\le l$ the induced action of $\bigrt_{q=1}^{j} F_{d_q}$ on the abelianization of $F_{d_k}$
is trivial. It is easy to check that $\cd(\Gamma)=l$.

Cohen {\cite[Thm.~3.1, Cor.~2.5, Lem.~3.3, Rem.~3.4]{Cohen}} gives the following description of the 
cohomology ring of the almost-direct product $\Gamma \;=\; \bigrt_{q=1}^{l} F_{d_q} $.
Let
$$  E \;=\; \Lambda\bigl(e_{i,p}\mid 1\le i\le L,\ 1\le p\le d_i\bigr)$$
be the exterior algebra on the degree-one classes dual to the generators of the free groups; thus
$E=\HH(\Z^{N})$, where $N=\sum_i d_i$ is the rank of $\Gamma^{\mathrm{ab}}$, the abelianization of $\Gamma$. 

\begin{theorem}[]\label{thm:cohen}
$\HH(\Gamma)\cong E/J$, where $J$ is the two-sided ideal generated by the quadratic
classes
$$  \eta_j^{p,q}\;=\;
  e_{j,p}e_{j,q}\;+\;\sum_{i=1}^{j-1}\sum_{r,s}\kappa^{p,q,r,s}_{i,j}\,e_{i,r}e_{j,s},
  \qquad 1\le j\le L,\ 1\le p<q\le d_j,
$$
whose coefficients are read off from the Fox calculus of the defining relations.
These classes form a quadratic Gr\"obner basis of $J$ for the degree-lexicographic
order and $E/J$ has the monomial basis
$$  \bigl\{\,\bar e_{Q}=\bar e_{1,Q_1}\cdots\bar e_{L,Q_L}  \ \bigm|\ |Q_i|\le 1\ \text{for each } i\,\bigr\}.$$
\end{theorem}

The leading term of $\eta_j^{p,q}$ is $e_{j,p}e_{j,q}$, so the standard monomials are
exactly those using at most one generator from each level.
 
\medskip
Fix $1\le m\le l-1$ and split off the top $m$ levels:
$$  K \;=\; \bigrt_{q=l-m+1}^{l} F_{d_q},   \qquad
  G \;=\; \bigrt_{q=1}^{l-m} F_{d_q},   \qquad   \Gamma \;=\; K\rt G . $$
The action of $G$ on the normal subgroup $K$ is in general non-trivial. Sections of
the projection $\Gamma\to G$ correspond to derivations $G\to K$; the trivial derivation
gives the canonical section $s_0\colon G\longrightarrow \Gamma$, given as the inclusion of $G$ as 
a complement of $K$. Write 
$T=\{e_{i,p}\mid i>l-m\}$ for the set of top-level classes of degree one.
 
\begin{lemma}\label{lem:kernel}
Under the identification of Theorem~\ref{thm:cohen}, $s_0^{*}\colon\HH(\Gamma)\to\HH(G)$
is the projection
$$  E_\Gamma/J_\Gamma\longrightarrow E_G/J_G,   \qquad
  e_{i,p}\longmapsto 
  \begin{cases}
    e_{i,p}, & i\le L-m,\\[2pt]
    0, & i>L-m,
  \end{cases}
$$
where $E_G=\Lambda(e_{i,p}\mid i\le L-m)$. Its kernel is the ideal $(T)$ generated by
the top-level classes.
\end{lemma}
\begin{proof}
Both rings are generated in degree one, so $s_0^*$ is determined by its restriction to
$H^1$, where it is dual to the map of abelianizations induced by $s_0$. Since the induced
actions on the abelianizations are t and $G$ is a
sub-semidirect product on the levels $i\le l-m$, this map is the coordinate
inclusion onto the summands with $i\le l-m$; dualizing gives the stated formula.
\end{proof}

The computation of $\tc(d_0)$ is based on the knowledge of the structural coefficients 
of the ideal generators $J$. Consider two explicit examples. 
 
\begin{example}[intermediate values]
Let $\Gamma=F_2\rt(F_2\rt F_2)$ (so $l=3$ and $d=(2,2,2)$) with $x_{1,1}$ acting on the
middle factor by $x_{2,1}\mapsto x_{2,2}^{-1}x_{2,1}x_{2,2}$, with $x_{2,1}$ acting on the
top factor by $x_{3,1}\mapsto x_{3,2}^{-1}x_{3,1}x_{3,2}$, and with all remaining actions
trivial. These assignments are compatible, so $\Gamma$ is indeed an almost-direct
product, and Cohen's algorithm yields
$$ \eta_1^{1,2}=e_{1,1}e_{1,2},  \qquad  \eta_2^{1,2}=e_{2,1}e_{2,2}-e_{1,1}e_{2,1},
  \qquad   \eta_3^{1,2}=e_{3,1}e_{3,2}-e_{2,1}e_{3,1}.$$
For $m=1$ we consider the derivation $d_0\colon F_2\rtimes F_2\to F_2$. Then $T=\{e_{3,1},e_{3,2}\}$, 
so $\cl(\ker s_0^*)\le2$, 
while $e_{3,1}e_{3,2}=e_{2,1}e_{3,1}$ is a standard monomial and hence non-zero. Thus
$\cl(\ker s_0^*)=2$, which is unfortunately smaller than the upper bound $l=3$. 
so $\tc(d_0)\in\{2,3\}$.

For $m=2$ we consider $d_0\colon F_2\to F_2\rtimes F_2$. Then 
$$  e_{2,2}\,e_{3,1}\,e_{3,2}   \;=\;e_{2,2}\,e_{2,1}\,e_{3,1}
  \;=\;-\,e_{1,1}\,e_{2,1}\,e_{3,1}\;\neq\;0 ,$$
so $\cl(\ker s_0^*)=3$ and we obtain a precise value $\tc(d_0)=3$.
\end{example}
 
\begin{example}[upper-triangular McCool group $P\Sigma_n^{+}$, see \cite{Cohen}, Example 4.7]
Upper triangular McCool group $P\Sigma_n^{+}$ is an important and much studied subgroup 
of the group of basis-conjugating automorphisms of the free group of rank n. It can be represented 
as $P\Sigma_n^{+}=\bigrt_{i=1}^{n-1}F_i$, so $l=n-1$ and $d_i=i$. It is thus somewhat similar to the 
pure braid group, however the actions involved are quite different. By  \cite[Ex.~4.7]{Cohen} its 
cohomology ring is described as in Theorem \ref{thm:cohen}, where 
the ideal $J$ is generated by the classes $e_{i,p}e_{j,i+1}-e_{j,p}e_{j,i+1}$ with $1\le p\le i<j\le n-1$. 

Taking $p=i=1,\dots n-1$ and $j=n-1$, and applying at each
step  the relation to the two leftmost factors lowers one factor per step:
$$  e_{n-1,1}e_{n-1,2}\cdots e_{n-1,n-1}   \;=\;e_{1,1}\,e_{n-1,2}\cdots e_{n-1,n-1}=$$
$$  \;=\;e_{1,1}e_{2,2}\,e_{n-1,3}\cdots e_{n-1,n-1}   \;=\;\cdots\;=\;e_{1,1}e_{2,2}\cdots e_{n-1,n-1}. $$
Each substitution replaces a degree-two factor, so that no sign changes occur. The final
outcome is a standard monomial, one generator per level, hence a non-zero element. All its $n-1$
factors are at the top level and therefore belong to $T$ for every $m\ge1$; this gives
$$   \cl (\ker s_0^*)\;=\;n-1\;=\;\cd P\Sigma_n^{+}=tc(s_0)   \qquad\text{for all } 1\le m\le n-2 .$$
\end{example}


\begin{thebibliography}{99}

\bibitem{ArkowitzStrom2004}
M.~Arkowitz and J.~Strom, \emph{The sectional category of a map},
Proc.\ Roy.\ Soc.\ Edinburgh Sect.\ A \textbf{134} (2004), no.~4, 639--652.

\bibitem{Bieri1977}
R.~Bieri, \emph{Homological dimension of discrete groups},
Queen Mary College Mathematics Notes, Queen Mary College, London, 1976.

\bibitem{Brown}
K.~S.~Brown, \emph{Cohomology of groups},
Graduate Texts in Mathematics 87, Springer-Verlag, New York, 1982.

\bibitem{Cohen}
D.~C.~Cohen, \emph{Cohomology rings of almost-direct products of free groups},
Compositio Math.\ \textbf{146} (2010), no.~2, 465--479.

\bibitem{CohenFarberWeinberger2021}
D.~C.~Cohen, M.~Farber and S.~Weinberger,
\emph{Topology of parametrized motion planning algorithms},
SIAM J.\ Appl.\ Algebra Geom.\ \textbf{5} (2021), no.~2, 229--249.

\bibitem{CohenFarberWeinberger2022}
D.~C.~Cohen, M.~Farber and S.~Weinberger,
\emph{Parametrized topological complexity of collision-free motion planning
in the plane},
Ann.\ Math.\ Artif.\ Intell.\ \textbf{90} (2022), no.~10, 999--1015.

\bibitem{CLOT}
O.~Cornea, G.~Lupton, J.~Oprea and D.~Tanr\'e,
\emph{Lusternik--Schnirelmann category},
Mathematical Surveys and Monographs 103,
American Mathematical Society, Providence, RI, 2003.

\bibitem{DeSahaDranishnikov2024}
A.~De Saha and A.~Dranishnikov,
\emph{On cohomological dimension of homomorphisms},
Proc.\ Amer.\ Math.\ Soc.\ \textbf{152} (2024), no.~11, 4607--4621.

\bibitem{DranishnikovKuanyshov2023}
A.~Dranishnikov and N.~Kuanyshov,
\emph{On the LS-category of group homomorphisms},
Math.\ Z.\ \textbf{305} (2023), no.~1, Paper No.~14, 12~pp.

\bibitem{EilenbergGanea1957}
S.~Eilenberg and T.~Ganea,
\emph{On the Lusternik--Schnirelmann category of abstract groups},
Ann.\ of Math.\ (2) \textbf{65} (1957), no.~3, 517--518.

\bibitem{EBFMO}
A.~Espinosa Baro, M.~Farber, S.~Mescher and J.~Oprea,
\emph{Sequential topological complexity of aspherical spaces and sectional
categories of subgroup inclusions},
Math.\ Ann.\ \textbf{391} (2025), no.~3, 4555--4605.

\bibitem{FadellHusseini2001}
E. Fadell, S. Husseini, \emph{Geometry and Topology of Configuration Spaces}, 
Springer Monographs in Mathematics, Springer 2001.

\bibitem{Garcia-Calcines2014}
J.~M.~Garc\'ia-Calcines,
\emph{Whitehead and Ganea constructions for fibrewise sectional category},
Topology Appl.\ \textbf{161} (2014), 215--234.

\bibitem{Giraud1971}
J.~Giraud, \emph{Cohomologie non ab\'elienne},
Grundlehren Math.\ Wiss.\ 179, Springer-Verlag, Berlin, 1971.

\bibitem{GrantPTC}
M.~Grant,
\emph{Parametrised topological complexity of group epimorphisms},
Topol.\ Methods Nonlinear Anal.\ \textbf{60} (2022), no.~1, 287--303.

\bibitem{GrantLuptonOprea2015}
M.~Grant, G.~Lupton and J.~Oprea,
\emph{A mapping theorem for topological complexity},
Algebr.\ Geom.\ Topol.\ \textbf{15} (2015), no.~3, 1643--1666.

\bibitem{GrantMeirPatchkoria2022}
M.~Grant, E.~Meir and I.~Patchkoria,
\emph{Equivariant dimensions of groups with operators},
Groups Geom.\ Dyn.\ \textbf{16} (2022), no.~3, 1049--1075.

\bibitem{Hatcher}
A.~Hatcher, \emph{Algebraic topology},
Cambridge University Press, Cambridge, 2002.

\bibitem{KasselTuraev2008}
C. Kassel, V. Turaev, \emph{Braid Groups}, GTM 257, Springer 2008.

\bibitem{Kuanyshov2023}
N.~Kuanyshov,
\emph{On the LS-category of homomorphisms of groups with torsion},
Algebra Discrete Math.\ \textbf{36} (2023), no.~2, 166--178;
erratum, ibid.\ \textbf{38} (2024), no.~1, 59--62.

\bibitem{Kuanyshov2024a}
N.~Kuanyshov,
\emph{On the LS-category of homomorphism of almost nilpotent groups},
Topology Appl.\ \textbf{342} (2024), Article ID 108776, 9~pp.

\bibitem{Kuanyshov2024b}
N.~Kuanyshov,
\emph{On the sequential topological complexity of group homomorphisms},
Topology Appl.\ \textbf{356} (2024), Article ID 109045, 18~pp.

\bibitem{NSW}
J.~Neukirch, A.~Schmidt and K.~Wingberg,
\emph{Cohomology of number fields},
Grundlehren Math.\ Wiss.\ 323, Springer-Verlag, Berlin, 2000.

\bibitem{OT}
P.~Orlik and H.~Terao, \emph{Arrangements of hyperplanes},
Grundlehren Math.\ Wiss.\ 300, Springer-Verlag, Berlin, 1992.

\bibitem{Pavesic2017}
P.~Pave\v si\'c, \emph{Complexity of the forward kinematic map},
Mech.\ Mach.\ Theory \textbf{117} (2017), 230--243.

\bibitem{Pavesic2019a}
P.~Pave\v si\'c, \emph{Topological complexity of a map},
Homology Homotopy Appl.\ \textbf{21} (2019), no.~2, 107--130;
see also \texttt{arXiv:1809.09021v3} for corrections to the printed version.

\bibitem{Pavesic2024}
P.~Pave\v si\'c, \emph{Topological complexity of a map},
in: M.~Farber and J.~Gonz\'alez (eds.), Topology and AI: topological aspects
of algorithms for autonomous motion,
EMS Ser.\ Ind.\ Appl.\ Math.\ 4, EMS Press, Berlin, 2024, pp.~335--362.

\bibitem{Rudyak2016}
Y.~B.~Rudyak,
\emph{On topological complexity of Eilenberg--MacLane spaces},
Topology Proc.\ \textbf{48} (2016), 65--67.

\bibitem{Scott2022}
J.~Scott, \emph{On the topological complexity of maps},
Topology Appl.\ \textbf{314} (2022), Article ID 108094.

\end{thebibliography}
\end{document}